\documentclass[12pt]{article}

\usepackage[all]{xy}
\usepackage{graphicx}
\usepackage{amssymb,latexsym,amsmath,amsthm, mathrsfs}
\usepackage{fullpage}
\usepackage{color}
\newtheorem{theorem}{Theorem}[section]
\newtheorem{lemma}[theorem]{Lemma}
\newtheorem*{theorem-non}{Theorem}

\newtheorem{definition}[theorem]{Definition}

\newtheorem{proposition}[theorem]{Proposition}
\newtheorem{corollary}[theorem]{Corollary}

 \DeclareMathSymbol{\N}{\mathbin}{AMSb}{"4E}

\DeclareMathSymbol{\Z}{\mathbin}{AMSb}{"5A}
\DeclareMathSymbol{\R}{\mathbin}{AMSb}{"52}
\DeclareMathSymbol{\Q}{\mathbin}{AMSb}{"51}
\DeclareMathSymbol{\I}{\mathbin}{AMSb}{"49}
\DeclareMathSymbol{\C}{\mathbin}{AMSb}{"43}

\def\M{{\mathbb M}}

\numberwithin{equation}{section}

\title{Real rectifiable currents, holomorphic chains and algebraic cycles }
\author{Jyh-Haur Teh, Chin-Jui Yang}
\date{}
\begin{document}
\maketitle

\begin{abstract}
We study some fundamental properties of real rectifiable currents and give a generalization of King's theorem in
characterizing currents defined by positive real holomorphic chains. Our proof uses Siu's semicontinuity
theorem and largely simplifies King's proof. A consequence of this result is a sufficient condition
for the Hodge conjecture.
\end{abstract}



\section{Introduction}
Since the publication of the foundational paper ``Normal and integral currents"(\cite{FF60})
by Federer and Fleming, geometric measure theory becomes an important tool in many areas of mathematics (\cite{F69}).
One particular fascinating question to us is to characterize currents defined by analytic varieties,
or more generally, by holomorphic chains, which are some formal linear combination of analytic varieties.
The first major progress was made by King in his marvelous paper \cite{K71} where he proved that holomorphic chains with
positive integral coefficients are those $d$-closed rectifiable positive currents. Three years later Harvey and Shiffman
improved King's result (\cite{HS74}) showing that $d$-closed rectifiable currents of type $(k, k)$ with $(2k+1)$-Hausdorff
measure 0 support are integral holomorphic chains. They also conjectured that the condition on support is not necessary.
This conjecture was resolved about twenty years later by Alexander (\cite{A97}). So for holomorphic chains
with integral coefficients, the characterization is complete. For holomorphic chains with real coefficients, as far as we know,
not much is known. A significant difference lies in the fact that for an integral current $T$, its density $\Theta(||T||, x)$
is a nonnegative integer, but for a real rectifiable current $R$, its density $\Theta(||R||, x)$, is a nonnegative real number (\cite{F74}).
For example consider the real rectifiable current $\sum^{\infty}_{n=1}\frac{1}{2^n}[\frac{1}{n}]$ where $[\frac{1}{n}]$ is the current defined by the point $\frac{1}{n}$.
This current is obviously a $d$-closed, positive, type $(0, 0)$ real rectifiable current but it can not be a holomorphic $0$-chain on $\C$
since the sequence $\{\frac{1}{n}\}^{\infty}_{n=1}$ has a limit point 0 and hence it can not be a holomorphic subvariety of $\C$.
This means that the 3 conditions: $d$-closedness, real rectifiability and positivity
are not sufficient to characterize positive real holomorphic chains. The main point of this paper is to show that in addition to the three conditions
mentioned above, for a current $R$ to be a positive real holomorphic chain, we need an extra condition that the set of positive density
$N:=\{x|\Theta(||R||, x)>0\}$ is of $\mathcal{H}^{2k}$-locally finite. Restriction on supports also appears naturally
in studying plurisubharmonic positive currents (\cite{DL03, HL09I, HL09II}).
We then show that currents with these four properties are positive real holomorphic chains. The following is our main result (Theorem \ref{generalization of King}).

\begin{theorem-non}
If $T \in RR^{loc}_{k,k}(U)$  is  positive, closed and $N = \{ x \in U : \Theta^{2k}( \| T \| , x ) > 0 \} $
is $\mathcal{H}^{2k}$-locally finite, then $T \in\mathscr{RZ}^{+}_k(U)$.
\end{theorem-non}

Since the last condition is automatically satisfied by positive integral currents, our result not only generalizes King's result but also our proof largely simplifies King's proof. This simplification is possible because of our use of Siu's famous semicontinuity theorem (\cite{S74}). Techniques from geometric measure theory are already important tools in studying algebraic cycles(\cite{L89, FL92, H77, L75}). In our opinion, our characterization of real holomorphic chains may find important applications in studying the Hodge conjecture. We give a sufficient condition for homology classes that can be represented by algebraic cycles with rational coefficients on
complex projective manifolds. In a forthcoming paper (\cite{TY2}), we propose a version of the Hodge conjecture in Bott-Chern cohomology and use a generalized
result of this paper to give a proof.

This paper is organized as follows. In section 2, we study some fundamental properties of real rectifiable currents.
This includes an integral representation for real rectifiable currents which plays an important role in later development.
We show that a locally normal current with $\mathcal{H}^{2k}$-locally finite support is actually real rectifiable.
In section 3, we give a generalization and at the same time, a new proof of King's theorem. In section 4, we apply our result
to give a sufficient condition for homology classes to be represented by algebraic cycles with rational coefficients.

%\begin{ACK}

%\end{ACK}

\section{Real rectifiable currents}
For $M$ a smooth oriented manifold, we denote by $A^r(M)$ the space of complex-valued smooth $r$-forms on $M$ and $A_c^r(M)$
the space of complex-valued $r$-forms with compact supports on $M$. Dually,
$\mathscr{D}^\prime_r(M)$ is the space of currents of dimension $r$ and $\mathcal{E}^\prime_r(M)$ is the space of currents with compact
supports.

\begin{definition}
Let $M$ be a smooth oriented manifold and $K \subset M$ be a compact set. A current $T\in \mathcal{E}'_r(M)$ is called a real rectifiable r-current on $M$ with support in $K$
if for every $\varepsilon>0$, there is
an open subset $U$ of some $\mathbb{R}^n$, a Lipschitz map $f : U \to M$ and a finite real polyhedral $r$-chain $P$ (in this article, we assume that simplices are nonoverlapping) with $f (\mbox{spt} P) \subset K$ such that the mass
$$\M(T - f_*(P))< \epsilon$$
Let $RR_{r, K}(M)$ be the space of real rectifiable $r$-currents on $M$ with supports in $K$ and the space of real rectifiable $r$-currents on $M$ is
$RR_r(M)=\bigcup_K RR_{r, K}(M)$, where the
union is taken over all compacta $K \subset M$. Locally real rectifiable $r$-currents on $M$ are elements of the set
$$RR_r ^{loc}(M):=\{T \in \mathscr{D}'_r(M) : \mbox{for $x \in M$,  there is $T_x \in RR_r(M)$ such that $x \notin spt(T - T_x)$}\}.$$
\end{definition}

We recall some definitions and results that we need later.

\begin{definition}
Suppose that $N$ is a $\mathcal{H}^n$-measurable subset of $\R^{n+k}$ and $\theta$ is a positive locally
$\mathcal{H}^n$-integrable function on $N$. We say that a given $n$-dimensional vector subspace $P$ of $\R^{n+k}$ is the
approximate tangent space for $N$ at $x$ with respect to $\theta$ if
for all $f \in A^0_c(\R^{n+k})$,
$$\lim_{\lambda\rightarrow 0} \lambda^{-n}
\int_N f(\lambda^{-1}(z-x))\theta(z)d\mathcal{H}^{n}(z) = \theta(x)\int_{P}f(y)d\mathcal{H}^{n}(y) $$
\end{definition}

The following result is \cite[Proposition 5.4.3]{KP08}.  We recall that a set $M\subset \R^{n+N}$
is countably $n$-rectifiable if there exists $n$-dimensional embedded $C^1$ submanifolds $N_1, N_2, .....$ and a set $N_0\subset \R^{n+N}$ with
Hausdorff $n$-measure $\mathcal{H}^n(N_0)=0$ such that
$$M\subset \bigcup^{\infty}_{k=0}N_k$$

\begin{proposition}\label{countably rectifiable}\label{same orient vector}
Suppose that $M\subset \R^{n+k}$ is $\mathcal{H}^n$-measurable and countably $n$-rectifiable. Then
$M=\bigcup_{j=0}^{\infty} S_j$
where
\begin{enumerate}
\item $\mathcal{H}^n(S_0) = 0$;
\item $S_i \bigcap S_j = \emptyset \mbox{ if } i\neq j$;
\item for $j \geq 1, S_j \subseteq N_j$ where $N_j$  is an  $n$-dimensional, embedded  $C^1$  submanifold  of
$\R^{n+k}$.
\end{enumerate}
\end{proposition}

Recall that a $\mathcal{H}^n$-measurable, $\Lambda_n(\R^{n+k})$-valued function $\xi$ is said to orient the approximate tangent space $T_xM$ of $M$ if there exist
orthonormal basis $\tau_1, ..., \tau_n$ of $T_xM$ such that
$$\xi(x)=\pm \tau_1\wedge \cdots \tau_n$$
for $\mathcal{H}^n$-a.e. $x\in M$.

The following result is \cite[Theorem 11.6]{L83}.

\begin{theorem}\label{approximate tangent plane}
Suppose that $N$ is $\mathcal{H}^n$-measurable. Then $N$ is countably $n$-rectifiable if and only if there is a positive
locally $\mathcal{H}^n$-integrable function $\theta$ on $N$ with respect to which the approximate tangent space $T_xN$ exists for
$\mathcal{H}^n$-almost every $x \in N$.
\end{theorem}

In the following, we generalize the integral representation theorem (\cite[Theorem 8.16]{FF60}) for integral currents to real rectifiable
currents. This result plays a fundamental role in the later development.

\begin{theorem}\label{representation}
If $T \in RR_k(\mathbb{R}^n)$, then for all $\varphi \in A_c^k(\mathbb{R}^n)$,
$$T(\varphi) = \int_W \quad \langle \ \varphi(x), \
\overrightarrow{T}(x) \ \rangle \ \theta(x) \ d\mathcal{H}^k \quad $$ where $W$ is countably
$k$-rectifiable, $\mathcal{H}^k$-measurable, $\theta:W \rightarrow \R$ is a positive $\mathcal{H}^k$-integrable function on $W, \|\overrightarrow{T}(x)\|=1$
and $\overrightarrow{T}(x)$ orients the approximate tangent space $T_xW$ for $\mathcal{H}^k$-almost every $x \in W$.
\end{theorem}

\begin{proof}
We follow the same strategy as Federer and Fleming proved \cite[Theorem 8.16]{FF60}. Let $C$ be the class of all $T \in \mathcal{E}'_k(\mathbb{R}^n)$
which has the integral representation as stated. Let $f: \R^k \rightarrow \R^n$ be a function and $U\subset \R^k$ be an open convex set.
The result will be proved in the following six steps:
\begin{enumerate}
\item If $f$ is continuously differentiable, $f$ is injective on the closure of $U$, and $Df(u)$ is injective for $u\in U$, then $f_*(U)\in C$.
\item If $T_i \in C$ for $i\in \N$, $\sum_{i=1}^{\infty}T_i=T \in \mathcal{E}'_k(\R^n)$ and
$\sum_{i=1}^{\infty}\M(T_i) < \infty$, then $T \in C$.
\item If $T_i\in C$ for $i\in \N$ and
$$\lim_{i\to \infty} T_i=T\in \mathcal{E}'_k(\R^n), \lim_{i\to \infty}\M(T_i-T)=0$$
then $T\in C$.
\item If $f$ is continuously differentiable, then $f_*(U)\in C$.
\item If $f$ is Lipschitzian, then $f_*(U)\in C$.
\item Every $k$-dimensional rectifiable current in $\R^n$ belongs to $C$.
\end{enumerate}

Except 2, the other statements are proved similarly to \cite[Theorem 8.16]{FF60}, so we only check the second statement in
the following.

Let $\varphi \in A_c^k(\mathbb{R}^n)$. By the assumption
$$T_i(\varphi) = \int_{W_i} \langle \ \varphi(x) \  , \  \overrightarrow{T}_i(x) \ \rangle \ \theta_i(x) \ d\mathcal{H}^k, \quad \M(T_i)=\int_{W_i}\theta_i(x)d\mathcal{H}^k$$
and
$$T(\varphi)=\sum^{\infty}_{i=1}T_i(\varphi) = \sum_{i=1}^{\infty}\int_{W_i} \
\langle \ \varphi(x) \  , \  \overrightarrow{T}_i(x) \ \rangle \ \theta_i(x) \ d\mathcal{H}^k $$

Let $W=\bigcup^{\infty}_{i=1}W_i$ and extend each $\theta_i$ by zero outside $W_i$.

Since $\varphi$ has compact support, we may assume $\| \varphi \| \leq 1$. Note that $\| \overrightarrow{T}_i(x)\|=1 $ for
$\mathcal{H}^k$-almost every $x \in W_i$.
So for any $n\in \N$,
$$\sum_{i=1}^{n} \ |<\varphi(x),  \overrightarrow{T}_i(x)\theta_i(x)>| \ \leq \ \sum_{i=1}^{\infty} \
\theta_i(x)$$

By the Lebesgue's dominated convergence theorem, we have
$$\int_{W} \sum_{i=1}^{\infty}
\theta_i(x) d\mathcal{H}^k=\sum_{i=1}^{\infty} \int_{W_i} \theta_i(x) d\mathcal{H}^k =\sum_{i=1}^{\infty}\M(T_i) < \infty$$
which implies $\sum_{i=1}^{\infty} \theta_i \in L^{1}(\mathcal{H}^k)$.

By the Lebesgue dominated convergence theorem again,
$$T(\varphi) = \int_{W}<\varphi(x), \nu(x)>d\mathcal{H}^k$$
where $\nu(x) = \sum_{i=1}^{\infty} \ \theta_i(x) \overrightarrow{T}_i(x)\in \Lambda^k(T_x\R^n)$ is convergent.

By the hypothesis, for each $i\in \N$, there is a subset $Y_i \subset W_i$ such that for every $x \in Y_i$,
$\overrightarrow{T}_i(x)$ exists, $\|\overrightarrow{T}_i(x)\|=1$, $\overrightarrow{T}_i(x)$ orients the approximate tangent space of $W_i$ at
$x$, and $\mathcal{H}^k(W_i\backslash Y_i)=0$. Let $Y=\bigcup_{i=1}^{\infty}Y_i$. Since $W$ is countably $k$-rectifiable and $\mathcal{H}^k$-measurable,
by Proposition \ref{countably rectifiable}, we may express $W$ as $W= \coprod_{j=0}^{\infty}S_j$ where $S_j \subseteq N_j$ for some $C^1$-manifold
$N_j$.

Let $Z=S_0 \bigcup (W\backslash Y)$. Then $\mathcal{H}^k(Z)=0$.
Fix $x \in W\backslash Z$. Then $x$ lies in some $N_j$. If $x \in W_i\bigcap W_l$, both $\overrightarrow{T}_i(x)$ and $\overrightarrow{T}_{\ell}(x)$
orient $T_xN_j$ which means
$\overrightarrow{T}_i(x) = \pm \overrightarrow{T}_{\ell}(x)$. Since $\sum_{i=1}^{\infty} \ \theta_i(x) \overrightarrow{T}_i(x) $
converges in $\bigwedge^k(T_xN_j)$,
$$\sum_{i=1}^{\infty} \ \theta_i(x) \overrightarrow{T}_i(x)=\sum^{\infty}_{i=1}\theta_i(x)a_i(x)\overrightarrow{T}_{\ell}(x)$$
where $a_i(x)=1$ or $-1$. Change all $a_i(x)$ to $-a_i(x)$ if necessary, we have
$$\theta:=\sum_{i=1}^{\infty} a_i\theta_i\geq 0$$
and $\sum_{i=1}^{\infty} \ \theta_i(x) \overrightarrow{T}_i(x)$
is of the form $\theta(x)\overrightarrow{T}(x)$
where $\overrightarrow{T}(x)=\pm \overrightarrow{T}_{\ell}(x)$.
Since $\sum^{\infty}_{i=1}\theta_i\in L^1(\mathcal{H}^k)$ and all $\theta_i's$ are nonnegative, $\theta\in L^1(\mathcal{H}^k)$.

For all $\varphi \in A_c^k(\R^n)$,
$$T(\varphi) = \int_W \quad \langle \ \varphi(x), \  \overrightarrow{T}(x) \ \rangle \ \theta(x) \ d\mathcal{H}^k $$
which means that $T \in C$. This completes the proof.
\end{proof}

The converse of the above result is also true.

\begin{theorem}\label{converse representation}
If $T \in \mathcal{E}'_k(\R^n)$ and
$$T(\varphi) = \int_W \ \langle \ \varphi(x), \  \overrightarrow{T}(x)
\ \rangle \ \theta(x) \ d\mathcal{H}^k$$
for all $\varphi \in A_c^k(\R^n)$ where $W$ is a countably
$k$-rectifiable and $\mathcal{H}^k$-measurable set, $\theta$ is a positive $\mathcal{H}^k$-integrable function on $\R^n$, $\| \overrightarrow{T}(x)\|=1$
and $\overrightarrow{T}(x)$ orients the approximate tangent space $T_xW$ for $\mathcal{H}^k$-almost every $x \in W$, then $T \in RR_k(\R^n)$.
\end{theorem}

Before we prove Theorem \ref{converse representation}, we need a simple result whose validity is rather clear.

\begin{lemma}\label{simplices}
Let $A$ be a bounded $\mathcal{L}^n$-measurable subset of $\mathbb{R}^n$. For any given $\varepsilon>0$,
there is a finite set of disjoint $n$-simplices which coincide with $A$ except for a set of measure less then
$\varepsilon$.
\end{lemma}

Recall that if $U$ is an open subset of $\R^k$ and $f:U\rightarrow \R^n$ is a $C^1$ map, then
$d^Uf_x:T_xU \rightarrow \R^n$ is defined by
$$d^Uf_x(\tau):=\sum^n_{j=1}<\tau, \nabla^Uf_j(x)>e_j$$
where $f=(f_1, f_2, ..., f_n)$ and $\nabla^Uf_j$ is the gradient of $f_j$ where $j=1, 2, ..., n$ and $\{e_1, ..., e_n\}$ is the standard
basis of $\R^n$.

Now we can prove Theorem \ref{converse representation}.

\begin{proof}
Since $T$ has compact support, we may assume that $W$ is bounded.  By Proposition \ref{countably rectifiable}, we may write
$W =\coprod_{j=0}^{\infty}S_j$ where all $S_j\subseteq N_j$ have properties as stated in Proposition \ref{countably rectifiable}.
We have
$$\M(T) = \int_{W} \theta(x) d\mathcal{H}^k(x) = \sum_{i=1}^{\infty} \int_{S_i} \theta(x) d\mathcal{H}^k(x) = \sum_{i=1}^{\infty}
\M(T\lfloor S_i) < \infty.$$
Given $\varepsilon > 0$. Choose $m \in \mathbb{N}$ such that
$$\sum_{i=m+1}^{\infty} \M(T\lfloor S_i) <\varepsilon.$$
Fix  $i$ with $1 \leq i \leq m$. Suppose that the $C^1$-manifold $N_i$ is parametrized by the $C^1$-diffeomorphism $f_i : U \rightarrow N_i$
for some open subset $U\subset \R^k$, and let $U$ be oriented by the natural orientation inherited from $\R^k$.
For $x \in S_i$, there is $y \in U$ such that $f_i(y) = x$. Define
\begin{displaymath}
\widetilde{\theta}_i(x) = \left\{ \begin{array}{ll}
\theta(x), & \textrm{if } d^Uf_{i, y*}(e_1\wedge \cdots \wedge e_k) \mbox{ and } \overrightarrow{T}(x) \mbox{ determine the same orientation.}\\
-\theta(x), & \textrm{otherwise.}\\
\end{array} \right.
\end{displaymath}

Let
$$\widetilde{T}_i(x) = \frac{d^Uf_{i, y*}(e_1\wedge \cdots \wedge e_k)}{\|d^Uf_{i, y*}(e_1\wedge \cdots \wedge e_k)\|}$$

Then
$$(T\lfloor S_i)(\varphi) = \int_{S_i} \langle \varphi(x) , \widetilde{T}_i(x) \rangle \widetilde{\theta}_i(x) d\mathcal{H}^k(x)$$

Let $\widehat{\theta}_i = \widetilde{\theta}_i \circ f_i$. Then $\widehat{\theta}_i$ is Lebesgue integrable. By the change of variable formula, we have
$T\lfloor S_i = f_{i\ast}(U\wedge\widehat{\theta}_i)$.
Given $\lambda_i > 0$.
Choose a simple function $\sum_{j=1}^Na^i_j\chi_{E^i_j}$ that is close to $\widehat{\theta}_i$ in $L^1$-norm
where $a^i_j \in \mathbb{R}$ and all $E^i_j \subset \R^n$ are Lebesgue measurable such that
$$\M(U\wedge \widehat{\theta}_i - U\wedge \sum_{j=1}^Na^i_j\chi_{E^i_j}) \leq ||\widehat{\theta}_i - \sum_{j=1}^Na^i_j\chi_{E^i_j}||_{L^1(U)} <\lambda_i$$

For each $j \in \{1, ... , N\}$, by Lemma \ref{simplices}, we can find finitely many disjoint polyhedrals $\Delta^i_{j, l}$ in $\R^n$ for $l = 1, ... ,q_j$  such that
$$\mathcal{H}^k((E^i_j\backslash \bigsqcup_{l=1}^{q_j}\Delta^i_{j, l}) \bigcup (\bigsqcup_{l=1}^{q_j}\Delta^i_{j, l}\backslash E^i_j)) < \frac{\lambda_i}{N|a^i_j|}$$

Let $P_i =\sum^N_{j=1}\sum_{l=1}^{q_j}a^i_j\Delta^i_{j, l}$.
Then
$$\M(U\wedge \sum^N_{j=1}a^i_j\chi_{E^i_j} - P_i) \leq \sum^N_{j=1}\int_{(E^i_j\backslash
\bigsqcup_{l=1}^{q_j}\Delta^j_l) \bigcup ( \bigsqcup_{l=1}^{q_j}\Delta^j_l\backslash E^i_j)} |a^i_j| d\mathcal{H}^k(x) <\lambda_i$$
This implies that
$$\M(T\lfloor S_i - f_{i_{\ast}}P_i) \leq  Lip(f_i)^k[ \M(U\wedge \hat{\theta}_i - U\wedge
\sum_{j=1}^Na^i_j\chi_{E^i_j}) + \M(U\wedge \sum^N_{j=1}a^i_j\chi_{E^i_j} - P_i) ] < 2 Lip(f_i)^k \lambda_i.$$
Now take $C = m(\max_{i=1,..., m} \{Lip(f_i)^k\})$ and $\lambda_i = \frac{\varepsilon}{2C}$. We
have
$$\M(T - \sum_{i=1}^mf_{i_{\ast}}P_i) \leq \M(\sum_{i=m+1}^{\infty}T\lfloor S_i) + \sum_{i=1}^{m}\M(T\lfloor S_i- f_{i_{\ast}}P_i) < 2\varepsilon$$
This completes the proof.
\end{proof}

\begin{definition}
\begin{enumerate}
\item A triple $(W, \theta, \overrightarrow{T})$ is called an oriented real $k$-rectifold if
$W$ is a countably $k$-rectifiable and $\mathcal{H}^k$-measurable set, $\theta$ is a positive
locally $\mathcal{H}^k$-integrable function on $W$, $\overrightarrow{T}(x)$ orients the approximate
tangent space $T_xW$ and $\| \overrightarrow{T}(x)\|=1$  for $\mathcal{H}^k$-almost every $x \in W$.

\item
The real rectifiable current  associated to an oriented real $k$-rectifold $(W, \theta, \overrightarrow{T})$
is the current $T\in RR_k^{loc}(\R^n)$ defined by
$$T(\varphi) = \int_W \quad \langle \ \varphi(x), \  \overrightarrow{T}(x) \ \rangle \ \theta(x) \ d\mathcal{H}^k$$
for $\varphi \in A_c^k(\R^n)$.
\end{enumerate}
\end{definition}

\begin{definition}
Let $U\subset \R^n$ be an open set. We say that a subset $A\subset U$ is of $\mathcal{H}^k$-locally finite if for
any $u\in U$, there is $r>0$ such that the Hausdorff measure $\mathcal{H}^k(A\cap B_r(u))<\infty$ where $B_r(u)$ is the open ball centered at $u$ with radius $r$.
\end{definition}

We denote by $\Omega(m)$ the volume of the $m$-dimensional unit closed ball.
\begin{definition}
Let $\mu$ be a measure on $\R^n$. The $m$-dimensional upper density of $\mu$ at $p$ is
$$\Theta^{\ast m} (\mu, p) := \limsup_{r\rightarrow 0} \frac{\mu[\overline{B_r(p)}]}{\Omega(m)r^{m}}$$
and the $m$-dimensional lower density of $\mu$ at $p$ is
$$\Theta_{\ast}^{m} (\mu, p) := \liminf_{r\rightarrow 0} \frac{\mu[\overline{B_r(p)}]}{\Omega(m)r^{m}}$$
If $\Theta^{\ast m} (\mu, p) = \Theta_{\ast}^{m} (\mu, p)$, then we call their common value the $m$-dimensional density of $\mu$ at $p$ and denote it by $\Theta^{m} (\mu, p)$.
\end{definition}

For a current $T$, we denote by $\|T\|$ the total variation of $T$.

\begin{proposition}\label{set with positive density}
Suppose that $T\in RR_k^{loc}(\R^n)$ is the real rectifiable $k$-current associated to an oriented real $k$-rectifold $(W, \theta, \overrightarrow{T})$. Let
$N = \{ x \in \R^n : \Theta^k( \| T \| , x ) > 0\}$ .
Then
$$T(\varphi) = \int_N \quad \langle \ \varphi(x), \  \overrightarrow{T}(x) \ \rangle \ \Theta^k( \| T \| , x ) \ d\mathcal{H}^k$$
for all $\varphi \in A_c^k(\R^n)$, and $W$ is $\mathcal{H}^k$-locally finite if and only if $N$ is $\mathcal{H}^k$-locally finite.
\end{proposition}

\begin{proof}
Let $\mu := \mathcal{H}^k\lfloor \theta=||T||$.
The existence of the approximate tangent plane (Theorem \ref{approximate tangent plane}) of $W$ implies that (see \cite[pg 63]{L83})
$$0<\theta(x) = \lim_{r\rightarrow 0^{+}} \frac{\int_{W\bigcap B_r(x)} \theta(y) d\mathcal{H}^k(y)}{\Omega(k) r^k} =
\Theta^k ( \mu , x)$$
for $\mu$-almost every $x\in W$. This implies that $\mathcal{H}^k(W-N)=0$.

Let
\begin{displaymath}
\Bar{\theta}(x) = \left\{ \begin{array}{ll}
\theta(x), & \textrm{if $x \in W$}\\
1, & \textrm{otherwise.}\\
\end{array} \right.
\end{displaymath}
and $\Bar{\mu} = \mathcal{H}^k \lfloor \Bar{\theta}$.

Since $\theta$ is locally $\mathcal{H}^k$-integrable, for each $x \in W$, we have
$$
\M(T\lfloor B_1(x)) = \int_{W\cap B_1(x)} \theta d\mathcal{H}^k < \infty.
$$
Since $\Bar{\mu}$ is Borel regular, $W\subset \R^n$ is $\Bar{\mu}$-measurable and
$$\Bar{\mu}(W\cap B_1(x))=\int_{W\cap B_1(x)}\Bar{\theta}(x)d\mathcal{H}^k(x)=\int_{W\cap B_1(x)}\theta(x)d\mathcal{H}^k(x)=\M(T\lfloor B_1(x))<\infty,$$
by \cite[Theorem 3.5]{L83},
$$\Theta^{*k}(\Bar{\mu}, W, y) = \Theta^{*k}(\Bar{\mu}, W\cap B_1(x), y) = 0$$
for $\Bar{\mu}$-almost every $y\in B_1(x)-W$.
Find a sequence $\{x_j\}_{j=1}^{\infty}$ in $W$ such that $W \subset V = \bigcup_{j=1}^{\infty}B_1(x_j)$. Then
$$\Theta^{*k}(\Bar{\mu}, W, y) = 0$$
for $\Bar{\mu}$-almost every $y\in V-W$.
Clearly, $\Theta^{*k}(\Bar{\mu}, W, y) = 0$ for all $y\in \mathbb{R}^n-V$.
Therefore
$\Theta^k(||T||, x)= 0$  for  $\mathcal{H}^k$-almost every $x\in \R^n- W$.
This implies $\mathcal{H}^k(N-W)=0$.

From the equality
$$\mathcal{H}^k [(W \setminus N) \bigcup (N \setminus W)] = 0$$
we may rewrite
$$T(\varphi)=\int_W<\varphi(x), \overrightarrow{T}(x)>d\mathcal{H}^k\lfloor \theta= \int_N <\varphi(x), \  \overrightarrow{T}(x)>\Theta^k( \| T \| , x ) \ d\mathcal{H}^k $$
for all $\varphi \in A_c^k(\mathbb{R}^n)$ and
we see that
$W$ is $\mathcal{H}^k$-locally finite if and only if $N$ is $\mathcal{H}^k$-locally finite.
\end{proof}

\begin{theorem}\label{normal to real rectifiable}
If  $T \in N^{loc}_k(\R^n)$ has $\mathcal{H}^k$-locally finite support, then $T$  is real rectifiable.
\end{theorem}

\begin{proof}
Let $S=\mbox{spt}(T)$. Since $S$ is $\mathcal{H}^k$-locally finite, by \cite[pg 494 (4)]{FF60},
$$ \Theta^k(\mathcal{H}^k, S, x) = \lim_{r\rightarrow 0^+} \frac{\mathcal{H}^k(S\bigcap B_r(x))}{\Omega(k)r^k} \geq 2^{-k}$$
for $\mathcal{H}^k$-almost every $x\in S$.
By \cite[Page 63]{L83}
\begin{align*}
\theta(x) &= \lim_{r\rightarrow 0^+} \frac{\| T \|(B_{r}(x))}{\mathcal{H}^k(S\bigcap B_r(x))}= \lim_{r\rightarrow 0^+} \frac{\| T \|(B_{r}(x))}{\Omega(k)r^k} \frac{\Omega(k)r^k}{\mathcal{H}^k(S\bigcap B_r(x))}\\
&\leq 2^k\lim_{r\rightarrow 0^+} \frac{\| T \|(B_{r}(x))}{\Omega(k)r^k}= 2^k\Theta^k(||T|| , x)
\end{align*}
for  $||T||$-almost every $x \in S$.
Thus
$\Theta^{k}(\| T \| , x) > 0$ for  $\| T \|$-almost every $x \in S$. Since $||T||(\R^n-S)=0$,
$\Theta^{k}(||T||, x)>0$ for $||T||$-almost every $x\in \R^n$. By \cite[Theorem 32.1]{L83}, $T$ is real rectifiable.
\end{proof}

\section{A generalization of King's theorem}

From now on, we let $U$ be an open subset of $\C^n$.

\begin{definition}
Let $\omega$ be the standard K\"{a}hler form of $\mathbb{C}^n$, $\omega_k=\frac{\omega^k}{k!}$ and $T \in \mathscr{D}^{\prime}_{k,k}(U)$ be a positive closed current. The Lelong
number $n(T , a)$ of $T$ at a point $a \in U$ is defined to be $\Theta^{2k}(T\wedge \omega_k , a)$.
\end{definition}

\begin{lemma}
If $T \in RR^{loc}_{k,k}(U)$ is positive and closed, then
$$n(T , a) = \Theta^{2k}(||T||, a)$$ for all $a\in U$. In particular, $\Theta^{2k}(||T||, a)$ exists for all $a\in U$.
\end{lemma}

\begin{proof}
By Wirtinger's inequality (\cite[Theorem 4.1]{F65}) and the integral representation of $T$ (Theorem \ref{representation}), for any Borel set $B \subset U$,
we have
$$(T\cap B)\wedge \omega_k=\int_B<\omega_k, \overrightarrow{T}>\theta d\mathcal{H}^{2k}=\int_B\theta d\mathcal{H}^{2k}=||T||(B)$$
 In particular, $\|T\|(B_r(a)) = (T\cap B_r(a))\wedge\omega_k=(T\wedge \omega_k)\cap B_r(a)$.
Hence $\|T\| = T\wedge \omega_k$, and therefore $n(T , a) = \Theta^{2k}(||T||, a)$.
Since Lelong number $n(T, a)$ exits for all $a\in U$, so does $\Theta^{2k}(T\wedge \omega_k , a)$.
\end{proof}

Since we have integral representation for locally real rectifiable currents. The following result is a simple modification of \cite[Lemma 1.12]{HS74}.
\begin{lemma}
Suppose that $T \in RR^{loc}_{2k}(U)$ is associated to the oriented real $2k$-rectifold $(W , \theta(x), \overrightarrow{T}(x))$.
Then $T \in RR^{loc}_{k, k}(U)$ if and only if $\overrightarrow{T}(x)$ is complex (i.e. $\overrightarrow{T}(x)$ represents a
complex subspace of $\mathbb{C}^n$) for $\|T\|$-almost every $x\in U$. Furthermore, $T \in RR^{loc}_{k,k}(U)$ is
positive if and only if $\overrightarrow{T}(x)$ is complex and positive for $\|T\|$-almost every $x\in U$.
\end{lemma}

We recall that a closed subset $A\subset X$ in a complex manifold $X$ is a holomorphic subvariety
if for any point $a\in A$, there is a neighborhood $W\subset X$ of $a$ and some $f_1, ..., f_m\in \mathcal{O}(W)$
such that
$$A\cap W=\{x\in X: f_1(x)=\cdots=f_m(x)=0\}$$

\begin{definition}
A current $T \in \mathscr{D}^{\prime}_{2k}(U)$ is said to be a real holomorphic $k$-chain on
$U$ if $T$ can be written in the form $T = \sum^{\infty}_{j=1} r_j [V_j]$ where $r_j \in \mathbb{R}$ and $V=\bigcup^{\infty}_{j=1} V_j$ is a
purely $k$-dimensional holomorphic subvariety of $U$ with irreducible components $\{V_j\}^{\infty}_{j=1}$. The vector space of
real holomorphic $k$-chains on $U$ is denoted by $\mathscr{RZ}_k(U)$. Also let $\mathscr{RZ}^{+}_k(U)$ denote the
set of positive real holomorphic $k$-chains on $U$, i.e., those real holomorphic $k$-chains with nonnegative coefficients.
\end{definition}

We recall the following semicontinuity theorem of Siu's (see \cite{D, S74}).
\begin{theorem}(Siu's semicontinuity  theorem)\label{Siu}
If $T$ is a closed positive current of bidimension $(k,k)$ on a complex manifold $X$, then the upperlevel sets
$$E_c(T) = \{ x \in X : n(T , x) \ \geq c \}$$
are holomorphic subvarieties of dimension $\leq k$.
\end{theorem}

\begin{proposition}\label{volume analytic set}
Let $U$ be an open subset of $\mathbb{C}^n$, and let $A_i, i = 1, 2, ...,$ be an  irreducible holomorphic
subvariety of dimension $k$ in $U$. If $A = \bigcup_{i=1}^{\infty}A_i$ is $\mathcal{H}^{2k}$-locally finite, then $A$
is a holomorphic subvariety of $U$.
\end{proposition}

\begin{proof}
Suppose that there is a point $a \in U$ such that each neighborhood of $a$ meets infinitely many $A_i's$.
Fix any $r > 0$ with $B_{4r}(a) \subset U$. Then $B_r(a)$ meets infinitely many $A_i's$.
Assume that $B_{r}(a)$ meets $A_{i_j}$, $j = 1, 2, ...$.
For each j, there is a point $a_j \in A_{i_j}$ such that $||a - a_j||< r$. Hence $B_r(a_j) \subset B_{4r}(a)$.
By \cite[Theorem B]{S66},
$$
\mathcal{H}^{2k}(B_{4r}(a)\cap A_{i_j}) \geq \mathcal{H}^{2k}(B_r(a_j)\cap A_{i_j})\geq \Omega(2k)r^{2k}
$$
Since
$$
\mathcal{H}^{2k}(B_{4r}(a)\cap A)=\lim_{m\to \infty}\mathcal{H}^{2k}(B_{4r}(a)\cap (\cup^m_{j=1}A_{i_j}))
$$
and $A_{i_j}\cap A_{i_k}$ is of $\mathcal{H}^{2k}$ measure 0 for $j\neq k$, by the inclusion-exclusion principle, we have
\begin{align*}
\mathcal{H}^{2k}(B_{4r}(a)\cap (\cup^m_{j=1}A_{i_j}))&=\mathcal{H}^{2k}(\cup^m_{j=1}(B_{4r}(a)\cap A_{i_j}))=\sum^m_{j=1}\mathcal{H}^{2k}(B_{4r}(a)\cap A_{i_j})\\
&\geq \sum_{j=1}^{m}\Omega(2k)r^{2k}=m\Omega(2k)r^{2k}
\end{align*}
which approaches $\infty$ as $m\rightarrow \infty$.
This is true for all $r$ sufficiently small, therefore $A$ is not $\mathcal{H}^{2k}$-locally finite which is a contradiction.
\end{proof}

\begin{proposition}\label{spt=N}
If $T \in RR^{loc}_{k}(U)$ and $N = \{ x \in U : \Theta^k( \| T \| , x ) > 0 \}$, then $\mbox{spt}(T) = \overline{N}$.
\end{proposition}

\begin{proof}
By Proposition \ref{set with positive density},
$$T(\varphi) = \int_N \quad \langle \ \varphi(x), \  \overrightarrow{T}(x) \ \rangle \ \Theta^k( \| T \| , x ) \ d\mathcal{H}^k$$
for all $\varphi \in A_c^k(U)$. If $a \notin \overline{N}$, there is a neighborhood $V$ of $a$ such that $V \cap \overline{N} = \emptyset$.
Hence for any $w \in A_c^k(U)$ with $\mbox{spt}(w) \subset V$, we have $T(w) = 0$.
Therefore $V \subset (\mbox{spt}(T))^c$, and this shows that $\overline{N}^c \subset (\mbox{spt}(T))^c$, equivalently, $\mbox{spt}(T) \subset \overline{N}$.
On the other hand, for $a \in N$, $\Theta^k( \| T \| , a ) > 0$ implies that there are infinitely many $r > 0$ such that
$$
\frac{\| T \| (B_r(a))}{\Omega(k)r^k} > 0
$$
For each such $r > 0$, since
$$
\| T \| (B_r(a)) = \mbox{sup}\{T(w) : w \in A_c^k(U) \ \mbox{ with } \ \|w\| \leq 1 \ \mbox{ and }\ \mbox{spt}(w) \subset B_r(a) \}
$$
we can find at least one $ w \in A_c^k(U)$ with $\mbox{spt}(w) \subset B_r(a)$ such that $T(w) > 0$.
This shows that $a \in \mbox{spt}(T)$, and hence $\overline{N} \subset \mbox{spt}(T)$ since $\mbox{spt}(T)$ is closed in $U$.
\end{proof}

We need the following result from \cite[Proposition 3.1.3]{K71}.
\begin{proposition}\label{flat chain subvariety}
If $V$ is a $k$-dimensional holomorphic subvariety of a complex manifold $X$, then for any
closed flat chain $T \in F^{loc}_{2k}(X)$ with $spt(T) \subset V$, $T$ is of the form $\sum^{\infty}_{j=1} a_j[V_j]$ where the
each $V_j$ is a global irreducible component of $V = \bigcup^{\infty}_{j=1} V_j $ and $a_j \in \mathbb{C}$.
\end{proposition}

Now we give a generalization of King's theorem to real rectifiable currents.

\begin{theorem}\label{generalization of King}
If $T \in RR^{loc}_{k,k}(U)$  is  positive, closed and $N = \{ x \in U : \Theta^{2k}( \| T \| , x ) > 0 \} $
is $\mathcal{H}^{2k}$-locally finite, then $T \in\mathscr{RZ}^{+}_k(U)$.
\end{theorem}

\begin{proof}
By Proposition \ref{set with positive density}, $N= \{x \in U : n(T , x) > 0 \}$.
Write
$$N = \bigcup_{n=1}^{\infty} E_n \mbox{ where } E_n = \{ x \in U : n(T , x) \geq \frac{1}{n} \}$$
Then by Siu's semicontinuity theorem, $E_n$ is a holomorphic subvariety of $U$ with dimension $\leq k$ for all $n \in \N$.
By the Measure Support Theorem (\cite[Theorem 2.4.2]{K71}), we may assume that each $E_n$ is of purely $k$-dimensional.
By the assumption, $N$ is of $\mathcal{H}^{2k}$-locally finite and hence by Proposition \ref{volume analytic set},
$N$ is a holomorphic subvariety and hence closed in $U$.
By Proposition \ref{spt=N}, $\mbox{spt}(T) = \overline{N} = N$.
Since $RR^{loc}_k(U) \subset F^{loc}_k(U)$, the result follows from Proposition \ref{flat chain subvariety}.
\end{proof}

In the following, we show that King's theorem is a simple consequence of our result.

\begin{corollary}
\begin{enumerate}
\item
Suppose that $T \in RR^{loc}_{k,k}(U)$  is  positive, closed and
$$N = \{ x \in U : \Theta^{2k}( \| T \| , x ) > 0 \}$$
is $\mathcal{H}^{2k}$-locally finite. If $n(T, x) \in \Z (\mbox{ respectively } \Q)$ for all $x\in U$, then $T$ is a holomorphic chain
with positive integral (respectively rational) coefficients.

\item (King's theorem) Suppose that $T\in R^{loc}_{k, k}(U)$ is positive and $d$-closed, then $T$ is a holomorphic chains with integral coefficients.
\end{enumerate}
\end{corollary}

\begin{proof}
\begin{enumerate}
\item By Theorem \ref{generalization of King}, $T=\sum^m_{j=1} a_j[V_j]$ is a holomorphic chain for some positive real numbers $a_j's$ and irreducible
$k$-dimensional holomorphic subvarieties $V_j's$ of $U$. For each $j$, choose $x_j\in V_j-\cap_{i\neq j}V_i$. Then
$n(T, x_j)=a_jn([V_j], x_j)$ and by Thie's theorem (\cite[Theorem 4.2.2]{K71}), $n([V_j], x_j)$ is a positive integer.
This implies that if $n(T, x_j) \in \Z (\mbox{ respectively } \Q)$, then $a_j \in \Z (\mbox{ respectively } \Q)$.

\item
By Lemma \cite[Lemma 1.14]{HS74}, $\mbox{spt}(T)$ has $\mathcal{H}^{2k}$-locally finite measure, so King's theorem follows
from (1).
\end{enumerate}
\end{proof}

\begin{corollary}
Let $U \subset \mathbb{C}^n$ be an open set. If $T \in RR^{loc}_{k,k}(U)$ is a closed, positive real rectifiable current and $n(T , a)$
is either $0$ or larger than $b$, where $b > 0$ for all $a \in U$, then $T$ is a holomorphic chain with real coefficients.
\end{corollary}

\begin{proof}
Because the result is local, we may restrict $U$ to a smaller open subset, also denoted by $U$, such that $\M(T) < \infty$.
Let $N = \{ a \in U : n(T , a) > 0 \}$. Then
$$ T(\varphi) = \int_{N} \ \langle \varphi(x) , \overrightarrow{T}(x) \rangle \ n(T , x) d\mathcal{H}^{2k}(x) $$
Hence
$$ b \mathcal{H}^{2k}(N) \leq \M(T) < \infty. $$
Therefore by Theorem \ref{generalization of King}, $T$ is a holomorphic chain with real coefficients.
\end{proof}

\section{Applications}
Let $X$ be a compact complex manifold. It follows from a well known result of Federer and Fleming \cite{FF60, F69} that the homology $H_*(\mathscr{D}'_{\bullet}(X; \R))$ of the chain complex $(\mathscr{D}'_{\bullet}(X; \R), d)$ is
isomorphic to the singular homology $H_*(X; \R)$ with real coefficients. In the following, $H_*(X; \R)$ denotes the homology of the chain complex $(\mathscr{D}'_{\bullet}(X; \R), d)$.

\begin{proposition}\label{R+dd^c}
Let $X$ be a complex projective manifold of complex dimension $n$ and $e\in A^{n-k, n-k}(X)$ is a $d$-closed form.
If $e$ considered as a current has the following property:
$$e=R+dd^cb$$
where $R$ is a current such that the $(k, k)$-part $R_{k, k}$ of $R$ is positive and $R_{k, k}$ has $\mathcal{H}^{2k}$-locally finite support, then
$e$ is homologous to some algebraic cycles with real coefficients.
\end{proposition}

\begin{proof}
Since $R_{k, k}$ is positive, it is normal. The support of $R_{k, k}$ is of $\mathcal{H}^{2k}$-locally finite and $dR_{k, k}=0$, by Theorem \ref{normal to real rectifiable},
$R_{k, k}\in RR_{k, k}(X)$, and by Theorem \ref{generalization of King} and Chow's Theorem, it is an algebraic cycle with real coefficients.
The result follows from the fact that $e=R_{k, k}+dd^cb_{k+1, k+1}$ where $b_{k+1, k+1}$ is $(k+1, k+1)$-part of $b$.
\end{proof}

\begin{theorem}
Let $X$ be a complex projective manifold of dimension $n$ and $e\in A^{n-k, n-k}(X)$.
If $e$ as a $k$-current is homologous to a Lipschitz $2k$-chain $P$ with rational coefficients which is $d$- and $d^c$-closed and the $(k, k)$-part of $P$ is positive, then $e$ is homologous to an algebraic cycle with rational coefficients.
\end{theorem}

\begin{proof}
By assumption we have
$$e=P+da$$
for some $a\in \mathscr{D}'_{2(k+1)}(X)$.
Since $d^cda=0$, by the $dd^c$-lemma, there is $b\in \mathscr{D}'_{2(k+1)}(X)$ such that
$$da=dd^cb$$
and hence
$$e=P+dd^cb$$

Since $\mathcal{H}^{2k}(\mbox{spt}(P))<\infty$, and by the assumption, its $(k, k)$-part is positive,
hence $P$ fulfills the hypothesis of Proposition \ref{R+dd^c} and therefore $e$ is homologous to an algebraic cycle
with real coefficients.

Note that by the assumption on $e$, $[e]\in H_{2k}(X; \Q)$.
Let
$$C_k(X; \Q)\subset H_{2k}(X; \Q), \ \ \ C_k(X; \R)\subset H_{2k}(X; \R)$$ be the subspaces generated by algebraic cycles with
rational and real coefficients respectively. Since we may find a basis for $C_k(X; \R)$ from algebraic cycles with integral
coefficients, these algebraic cycles also form a basis for $C_k(X; \Q)$, then we have the equality
$$C_k(X; \Q)=H_{2k}(X; \Q)\cap C_k(X; \R)$$
Applying the above observation to our case, we have $[e]\in H_{2k}(X; \Q)\cap C_k(X; \R)$, and hence in $C_k(X; \Q)$.
\end{proof}

In the forthcoming paper \cite{TY3}, we formulate a version of the Hodge conjecture in Bott-Chern cohomology and combine results
from this paper and results from \cite{TY2} to prove it.

\bibliographystyle{amsplain}

\begin{itemize}
\item[] Jyh-Haur Teh, Department of Mathematics, Tsing Hua University(Hsinchu), 101 KuangFu Road, Section 2, 30013 Hsinchu, Taiwan.\\
Email address: jyhhaur@math.nthu.edu.tw

\item[] Chin-Jui Yang, Department of Mathematics, Tsing Hua University(Hsinchu), 101 KuangFu Road, Section 2, 30013 Hsinchu, Taiwan.\\
Email address: p9433256@gmail.com
\end{itemize}

\end{document}